\documentclass[a4paper,11pt,makeidx]{amsart}
\usepackage{amscd}
\usepackage{xypic}  
\usepackage{amssymb}
\usepackage{amsthm}
\usepackage{epsf}
\makeindex

\newtheorem{thm}{thm}[section]
\newtheorem{theorem}[thm]{Theorem}
\newtheorem{corollary}[thm]{Corollary}
\newtheorem{example}[thm]{Example}
\newtheorem{lemma}[thm]{Lemma}

\newtheorem{prop}[thm]{Proposition}

\newtheorem{remark}[thm]{Remark}
\newtheorem{definition}[thm]{Definition}

\newtheorem{question}[thm]{Question}
\newtheorem{thevarthm}[thm]{\varthmname}

\newenvironment{varthm*}[1]{\trivlist\item[]{\bf #1.}\it}{\endtrivlist}

\def\c1{\operatorname{c_1}}
\def\c2{\operatorname{c_2}}

\def\CC{{\mathbb C}}

\def\*{\otimes}

\def\+{\oplus}                   
\def\*{\otimes}                  

\newcommand\eqnref[1]{(\ref{#1})}

\newcommand\lra{\longrightarrow}

\def\endproof{\hspace*{\fill}\endproofsymbol\endtrivlist}

\def\endproofsymbol{\frame{\rule[0pt]{0pt}{6pt}\rule[0pt]{6pt}{0pt}}}

\begin{document}

\title{Geometry of the locus of polynomials of degree $4$ with iterative roots}
\author{Beata Strycharz-Szemberg, Tomasz Szemberg}

\address{\hskip -.43cm Beata Strycharz-Szemberg,
   Instytut Matematyki, Politechnika Krakowska, Warszawska 24, PL-31-155 Krak\'ow, Poland}
\email{szemberg@pk.edu.pl}

\address{\hskip -.43cm Tomasz Szemberg,
   Instytut Matematyki UP,
   Podchora\.zych 2, PL-30-084 Krak\'ow, Poland}
\curraddr{Albert-Ludwigs-Universit\"at Freiburg,
   Mathematisches Institut,
   Eckerstra{\ss}e 1,
   D-79104 Freiburg,
   Germany
}
\email{tomasz.szemberg@uni-due.de}

\begin{abstract}
   We study polynomial iterative roots of polynomials and describe
   the locus of complex polynomials of degree $4$ admitting a polynomial iterative
   square root.
\end{abstract}

\maketitle

\section{Introduction}

   We begin by recalling some basic notions about iterations.
   Let $M$ be an arbitrary set and let $f:M\lra M$ be a function.
\begin{definition}
   The \emph{iterates} of $f$ are defined recursively by
   \begin{itemize}
   \item $f^0(x)=x$ for all $x\in M$;
   \item $f^{n+1}(x)=f(f^n(x))$ for all $x\in M$ and $n\geq 0$.
   \end{itemize}
\end{definition}
\begin{example}\label{iteracja_liniowego}
   Let $f:\CC\lra\CC$ be a linear polynomial $f(z)=az+b$. Then
   \begin{itemize}
   \item[] $f^0(z)=z$;
   \item[] $f^1(z)=az+b$;
   \item[] $f^2(z)=a(az+b)+b=a^2z+ab+b$;
   \item[] $f^3(z)=a^3z+a^2b+ab+b$;
   \item[] $\vdots$
   \item[] $f^n(z)=a^nz+(a^{n-1}+a^{n-2}+\dots+a+1)b$.
   \end{itemize}
\end{example}
   It is natural to ask if a given function $g$ can be represented as an iterate
   of another function $f$. If this is the case, then we say that $f$ is an iterative
   root of $g$. This is defined precisely below.
\begin{definition}
   We say that a function $f$ is an \emph{iterative root} of order $r\geq 2$ of a function~$g$, if
   the following functional equation is satisfied
   $$f^r=g$$
   and $r$ is the least integer for which this equation holds.
\end{definition}
   The problem of the existence of iterative roots appear in the literature
   already at the beginning of the 19th century, notably in the works of
   Abel and Babbage, see eg. \cite{Bab}.
   In the present note we consider this problem for complex polynomials
   in one variable.
   Iterations of polynomials of more variables were studied recently in \cite{CSZ}.
\section{Complex polynomials and algebraic sets}

   We recall some notions necessary in the further considerations.
\begin{definition}
   a) A \emph{complex polynomial} $g(z)$ of one variable $z$ of degree $d$ is an expression
   of the form
   $$g(z)=b_dz^d+b_{d-1}z^{d-1}+\dots+b_1z+b_0,$$
   with $b_0,\dots,b_d\in\CC$ and $b_{d}\neq 0$. If $d=0$, then also $b_0=0$ is allowed.

   b) The \emph{set} $P(d)$ \emph{of polynomials} of degree $\leq d$ is a vector space of dimension
   $d+1$ with natural coordinates $(b_d,b_{d-1},\dots,b_0)$.

   c) The set of \emph{normalized} polynomials $P_n(d)$ of degree $d$ (i.e. those
   with $b_d=1$) is an affine subspace of $P(d)$, which we identify with $\CC^d$
   via natural coordinates $(b_{d-1},b_{d-2},\dots,b_0)$.
\end{definition}
   Complex polynomials in more variables are defined similarly. Their set has
   a structure of ring and is denoted usually by $\CC[z_1,\dots,z_n]$.
\begin{definition}
   An \emph{(affine) algebraic set} in $\CC^n$ is a set $X$ defined by a (finite) number of polynomial
   equations
   $$X=\left\{\begin{array}{ccc}
      f_1(z_1,\dots,z_n) & = & 0\\
      f_2(z_1,\dots,z_n) & = & 0\\
      \vdots & \vdots & \vdots \\
      f_k(z_1,\dots,z_n) & = & 0\\
      \end{array}\right.$$
      where $f_1,\dots,f_k\in\CC[z_1,\dots,z_n]$.
\end{definition}
   One of fundamental theorems in algebraic geometry is Hilbert's Nullstellensatz \cite[Theorem 1.16]{Hul}.
   It associates to any algebraic subset $X\subset\CC^n$ its vanishing ideal $I(X)$. The transcendence degree
   of the quotient field $\CC(X)$ of the quotient ring $\CC[z_1,\dots,z_n]/I(X)$ is the \emph{dimension}
   of $X$.

\section{Iterative roots of complex polynomials}
   Motivated by the ideas recalled in the Introduction we are led to the following problem.

\begin{question}\label{question}
   Do complex polynomials in one variable of given degree $d$ admit iterative roots?
\end{question}

   This question is very general and a little bit ambiguous at the same time.
   First of all note, that if $f$ is a polynomial of degree $e$, then $f^r$
   is a polynomial of degree $e^r$, so there are obvious constrains involving
   the degree for the existence of \emph{polynomial} iterative roots.
   For example a polynomial whose degree $d$ is a prime cannot admit any
   polynomial roots with $r\geq 2$.

   Sending a polynomial $f\in P(d)$ to its $r-$th iterate $f^r\in P(d^r)$
   is an algebraic mapping. This means the coefficients of $f^r$ are polynomial
   expressions in coefficients of $f$. The image is thus an algebraic set
   and its dimension is bounded by $d+1$, the dimension of $P(d)$, whereas the dimension of $P(d^r)$
   is $(d^r+1)$.
   This implies that a general polynomial of degree $d^r$ has no polynomial
   root of order $r$. It is interesting to have some criteria.
   This motivated our Theorem \ref{main}.

   On the other hand, one might wonder if given a polynomial $g(z)$ there
   exists \emph{arbitrary} function such that
   $$f^r=g$$
   with no assumptions whatsoever on the regularity of $f$.
   This question seems to be much harder. We collect together
   what is known about Question \ref{question} to put our result
   in a perspective.

\subsection{Constant polynomials}
   A polynomial $g$ of degree $0$ is a constant function $g(z)=b_0$.
   It is clear that $g$ is its own iterative root of arbitrary degree
   i.e. $g^r=g$ for all $r\geq 1$.

\subsection{Linear polynomials}
   A linear polynomial $g:\CC\ni z\lra az+b\in\CC$ with $a\neq 0$ has iterative
   roots of arbitrary order $r$.
\proof
   If $a\neq 1$, then
   set simply $c:=a^{\frac1r}$ (we take one of the roots of order $r$ of $a$),
   $d:=\frac{b}{1+c+\dots+c^{r-1}}=b\frac{1-c}{1-a}$ and define $f(z)=cz+d$. Then $f$ is an iterative
   root of $g$ of order $r$, compare Example \ref{iteracja_liniowego}.

   If $a=1$ i.e. $g(z)=z+b$, then clearly $f(z)=z+\frac{b}{r}$ is the iterative of $g$ of order $r$.

   If $b=0$, then for $\alpha$ an arbitrary root of $1$ of order $r$ the polynomial
   $f(z)=\alpha\cdot z$ is an iterative root of $g(z)=z$.
\endproof
   It follows from the proof that a complex linear polynomial admits at least $r$ distinct
   iterative roots of order $r$, which are linear polynomials as well. One cannot
   however hope for an iterative version of the fundamental theorem of algebra. Indeed,
   the identity $g(z)=z$ has infinitely many polynomial iterative square roots of the form
   $f(z)=-z+b$ with arbitrary $b\in\CC$.

\subsection{Quadratic polynomials}
   As the degree $d=2$ is a prime, there cannot exist any polynomial iterative roots.
   On the other hand in \cite{RSS} Rice, Schweizer and Sklar proved a surprising result
   to the effect that a quadratic complex polynomial does not admit \emph{any} iterative
   root i.e., with no assumption on the regularity of a root.

\begin{theorem}[Rice-Schweizer-Sklar]\label{quadratic}
   Let $g(z)$ be a quadratic complex polynomial and $r\geq 2$ an integer. Then
   there does not exist any function $f:\CC\lra\CC$ such that $f^r=g$.
\end{theorem}

\begin{remark}\label{real poly}\rm
   Note that the result strongly depends on the ground field. For example
   there exist iterative square roots of some real quadratic polynomials, see \cite{B}
   for details.
\end{remark}

\subsection{Cubic polynomials}
   The degree is again a prime, so there are no polynomial roots.
   Choczewski and Kuczma claim in \cite{CK} that the degree $3$ counterpart
   of Theorem \ref{quadratic} holds. However the factorization in
   \cite[Formula 20]{CK} is false and the rest of their argument fails, so that
   we don't know for sure.

\subsection{Quartic polynomials}
   Obviously there are quartic polynomials which are iterative squares of quadratic
   polynomials. As remarked already above a general quartic polynomial is not
   an iterative square, so it is natural to ask how to distinguish those which are
   among all the others.
   In this note we show that there exists a rational surface $S$ in the space
   $P_n(4)$ of normalized polynomials of degree $4$, which parametrizes
   iterative squares. The precise statement is given in Theorem \ref{main},
   which is our main result.
   Degree $4$ is in a sense the first interesting case because there are quartic
   polynomials which have polynomial iterative roots of order $2$
   and such which haven't.

\subsection{An overview}
   We summarize facts presented in this section in the following table.\smallskip

\renewcommand{\arraystretch}{1.2}
\begin{tabular}{|c|c|c|}
   \hline
   degree & which have polynomial iterative roots & which have iterative roots \rule[-2ex]{0mm}{5ex} \\
   \hline
   0 & all & all \\
   1 & all & all \\
   2 & none & none \\
   3 & none & none (?)\\
   4 & a proper algebraic subset & ???\\
   \hline
\end{tabular}
\renewcommand{\arraystretch}{1}\medskip

   The following problem is quite natural (the evidence might be to sparse to rise
   a conjecture). We stress again, that in view of Remark \ref{real poly} it is
   crucial, that we are over the complex numbers.
\begin{question}
   Let $g$ be a complex polynomial in one variable which has an iterative root
   of order $r$. Does there exist a \emph{polynomial} iterative root of $g$ of order $r$?
\end{question}
   Note that one cannot hope that any iterative root of a polynomial is
   a polynomial itself. For example the identity $g(z)=z$ has also rational
   functions $f(z)=\frac{a}{z}$ with arbitrary $a\in\CC$, $a\neq 0$ as iterative square
   roots (and more complicated functions as well).

\section{The main result}
   We begin by showing that it is enough to consider normalized polynomials i.e. to
   work in the space $P_n(4)$. Let $L(z)=az+b$ be a linear polynomial, $a\neq 0$.
   Then $L$ is a bijective mapping of the complex plane with the inverse $L^{-1}(z)=\frac{1}{a}z-\frac{b}{a}$.
   We have the following observation.
\begin{lemma}\label{linear conjugate}
   Let
   $$g(z)=b_dz^d+b_{d-1}z^{d-1}+\dots+b_1z+b_0$$
   be a complex polynomial of degree $d$. There exists a normalized polynomial $g_n$, which
   is linear conjugate to $g$ i.e. $g_n=L^{-1}\circ g\circ L$ for some $L(z)=az+b$ as above.
\end{lemma}
\proof
   Elementary calculation shows that
   $$(L^{-1}\circ g\circ L)(z)=(b_d\cdot a^{d-1})z^d+\mbox{ terms of lower degree}.$$
   Since we work over the complex numbers there exists $a\in\CC$ such that
   $$b_d\cdot a^{d-1}=1$$
   and the assertion of the Lemma follows.
\endproof

   It is a general fact, that existence of iterative roots is invariant under
   conjugation. Indeed, assume that
   $$\gamma=\varphi^r$$
   holds and let $\lambda$ be any bijection. Then $\lambda^{-1}\circ\gamma\circ\lambda$
   has $\lambda\circ\varphi\circ\lambda$ as its iterative root of order~$r$.
   Hence it is enough to consider normalized polynomials.

   Let $f(z)=z^2+a_1z+a_0$ be a normalized quadratic polynomial. Then
   $$f^2(z)=z^4+2a_1\cdot z^3+(2a_0+a_1^2+a_1)\cdot z^2+(2a_1a_0+a_1^2)\cdot z+(a_0^2+a_1a_0+a_0).$$
   This calculation can be interpreted in terms of the polynomial mapping
   \begin{equation}\label{odwzfi}
   \varphi:P_n(2)=\CC^2\ni(a_1,a_0)\to
     \left\{\begin{array}{rcl}
        b_3 & = & 2a_1\\
        b_2 & = & 2a_0+a_1^2+a_1\\
        b_1 & = & 2a_1a_0+a_1^2\\
        b_0 & = & a_0^2+a_1a_0+a_0
        \end{array}\right.\in\CC^4=P_n(4).
   \end{equation}
   We denote the image of this mapping by $S$.

   Eliminating the variables $a_0$ and $a_1$ in the above equations
   either by hand or using Singular \cite{DGPS} as we did,
   leads to
   the following equations for $S$:
   \begin{align}
      b_3^4-8b_2b_3^2-12b_3^2+16b_2^2+32b_2-16b_3-64b_0 &=0; \label{equation_special}\\
      b_3^3-4b_2b_3+8b_1 &=0 \notag
   \end{align}
   Up to rearranging the terms we just proved our main result.
\begin{theorem}\label{main}
   Let $g(z)=z^4+b_3z^3+b_2z^2+b_1z+b_0$ be a normalized complex quartic polynomial.
   There exists a normalized polynomial $f(z)=z^2+a_1z+a_0$ of degree $2$ such that $f^2=g$ if and only if
   $$b_0=\frac{1}{64}(b_3^4-8b_2b_3^2-12b_3^2+16b_2^2+32b_2-16b_3)\;\; \mbox{ and }\;\;
     b_1=-\frac18(b_3^3-4b_2b_3).$$
   In this case
   \begin{equation}\label{effroot}
      a_1=\frac12 b_3\;\mbox{ and }\, a_0=-\frac18b_3^2-\frac14b_3+\frac12b_2.
   \end{equation}
\end{theorem}
\begin{remark}\rm
   Whereas equations \eqnref{equation_special} distinguish iterative squares exactly in the same
   manner as for example the vanishing of the discriminant of a polynomial distinguishes polynomials with multiple roots,
   the forms \eqnref{effroot} show that a polynomial square root can be easily computed effectively.
\end{remark}
\begin{remark}\label{non normalized}\rm
   As pointed out by the referee it might happen that a \emph{normalized} polynomial
   has an iterative root which is \emph{not normalized}. Indeed, if $\alpha$ is an arbitrary
   cubic root of $1$, then the second iterate of $h(z)=\alpha z^2+\beta_1 z+\beta_0$
   is a \emph{normalized} polynomial of degree $4$.
\end{remark}
   As we saw in Lemma \ref{linear conjugate}, the polynomials $h(z)$ in Remark
   \ref{non normalized} and $f(z)$ in Theorem \ref{main} are linearly conjugate,
   for example by $L_{\alpha}(z)=\alpha\cdot z$ which sends
   $$f(z)=z^2+a_1z+a_0\,\,\mbox{ to }\,\, h(z)=\alpha z^2+a_1z+\alpha^2 a_0,$$
   i.e. $\beta_1=a_1$ and $\beta_0=\alpha^2 a_0$.
   The second iterates of the above polynomials are normalized polynomials, of course
   again conjugate by $L_{\alpha}$.

   From now on, we assume that $\alpha$ is a primitive cubic root.
   Starting with $h(z)$ instead of $f(z)$ we get a mapping $\varphi_{\alpha}$
   similarly as in \eqnref{odwzfi}
   $$
   \varphi_{\alpha}:P_n(2)=\CC^2\ni(a_1,a_0)\to
     \left\{\begin{array}{rcl}
        b_3 & = & 2\alpha^2 a_1\\
        b_2 & = & 2\alpha^2 a_0+\alpha a_1^2+\alpha a_1\\
        b_1 & = & 2\alpha a_1a_0+a_1^2\\
        b_0 & = & \alpha a_0^2+a_1a_0+a_0
        \end{array}\right.\in\CC^4=P_n(4).$$
   and by the same token we have the map $\varphi_{\alpha^2}$. We denote the images of these maps
   by $S_{\alpha}$ and $S_{\alpha^2}$ respectively. It is clear that
   the conjugation by $L_{\alpha}$, respectively by
   $L_{\alpha^2}(z)=\alpha^2\cdot z$ induces isomorphisms of $S$ with
   $S_{\alpha}$ and $S_{\alpha^2}$. It suffices to understand the geometry
   of $S$. It is surprisingly simply.
   Since $S$ is the graph
   of a polynomial mapping in variables $b_2$ and $b_3$, it is a rational surface.
   It has an isolated singularity in the point $(0,0,0,0)$ which corresponds to the
   polynomial $z^4$. This polynomial sits in the intersection
   $C:=S\cap S_{\alpha}\cap S_{\alpha^2}$, hence it has three distinct polynomial
   square roots: $z^2$, $\alpha\cdot z^2$ and $\alpha^2\cdot z^2$.
   It is natural to ask if there are more such polynomials.
\section{Multiple roots and geometry}
   Here we show that $C$ is an irreducible rational curve, hence
   in particular there is a one parameter family of quartic polynomials
   with triple iterative roots of order $2$.
\begin{prop}\label{multi roots}
   Let $g(z)=z^4+\beta\cdot z^3+\frac{3}{8}\beta^2\cdot z^2+\frac{1}{16}\beta^3\cdot z+\frac{1}{256}\beta^4-\frac{1}{4}\beta$
   with $\beta\in\CC$ arbitrary. Then $g$ has three distinct iterative roots of order $2$.
\end{prop}
\begin{proof}
   A naive idea to describe all polynomials with multiple iterative roots would be to
   consider pairwise intersections of surfaces $S$, $S_{\alpha}$ and $S_{\alpha^2}$.
   This is doable but the computations are nasty. Here we compute instead iterates
   of quadratic polynomials.
   Let
   $$f_{\alpha}(z)=\alpha z^2+b_1z+b_0,$$
   then we have
   \begin{equation}\label{second iterate}
      f^2(z)=\alpha^3 z^4+ 2\alpha^2 b_1 z^3+(2\alpha^2 b_0+\alpha b_1^2+\alpha b_1)z^2+(2\alpha b_0b_1+b_1^2)z+(\alpha b_0^2+b_0b_1+b_0).
   \end{equation}
   It is clear that $f^2(z)$ is a normalized polynomial if and only if $\alpha$ is a cubic root of $1$. We assume
   that $\alpha$ is a primitive root. We write
   $$f_1(z)=z^2+a_1z+a_0,\;\;\mbox{ and }\; f_{\alpha^2}(z)=\alpha^2 z^2+c_1z+c_0.$$
   We get formula for $f_1^2(z)$, respectively $f_{\alpha^2}^2(z)$
   replacing in \eqnref{second iterate} $\alpha$ by $1$, respectively by $\alpha^2$.

   With second iterates of $f_1$, $f_{\alpha}$ and $f_{\alpha^2}$ computed, we compare coefficients
   at corresponding powers of $z$. We begin with $f_1$ and $f_{\alpha}$, which corresponds to computing
   the intersection of $S$ with $S_{\alpha}$. We get the following system of equations
   $$\left\{\begin{array}{rcl}
      a_1 & = & \alpha^2 b_1\\
      2a_0+a_1^2+a_1 & = & 2\alpha^2b_0+\alpha b_1^2+\alpha b_1\\
      2a_1a_0+a_1^2 & = & 2\alpha b_0b_1+b_1^2\\
      a_0^2+a_1a_0+a_0 & = & \alpha b_0^2+b_0b_1+b_0
   \end{array}\right.$$
   Using elimination (we were again aided by Singular but it can be done also directly) we get
   $$\left\{\begin{array}{ccl}
      b_1 & = & \alpha a_1\\
      b_0 & = & \frac14\alpha(a_1-2\alpha^2) a_1\\
      a_0 & = & \frac14(a_1-2) a_1
   \end{array}\right.$$
   With $\beta:=2a_1$ we can write two iterative square roots of $g(z)$ explicitly as
   $$\begin{array}{ccl}
      f_1(z) & = & z^2+\frac12\beta z-\frac{1}{16}\beta(\beta\alpha^2+\beta\alpha+4),\\
      f_{\alpha}(z) & = & \alpha z^2+ \frac12\beta\alpha z-\frac{1}{16}\beta(\beta+\beta\alpha^2+4).
      \end{array}$$
   The intersection curve $\Gamma:=S\cap S_{\alpha}$ is thus parametrized by
   the polynomial mapping
   $$\CC\ni\beta\to \left(\beta,\frac{3}{8}\beta^2,\frac{1}{16}\beta^3,\frac{1}{256}\beta^4-\frac{1}{4}\beta\right)\in P_n(4)=\CC^4.$$
   In particular this curve is rational and irreducible.
   It is not hard to check in the same way,
   that $S_{\alpha^2}$ intersects $S$ and $S_{\alpha}$ exactly along $\Gamma$.
   Alternatively one could argue with the cyclic group of order $3$ generated by $\alpha$
   acting on the whole picture.
   In any case $\Gamma=C$ and the third iterative root of $g$ is
   $$\begin{array}{ccl}
      f_{\alpha^2}(z) & = & \alpha^2 z^2 + \frac12\beta\alpha^2 z-\frac{1}{16}\beta(\beta+\beta\alpha+4).
      \end{array}$$
\end{proof}
   Proposition \ref{multi roots} has the following somewhat surprising consequence.
\begin{corollary}
   A normalized quartic polynomial has
   \begin{itemize}
   \item[a)] either no polynomial iterative square root (it lies in the complement of the set $S\cup S_{\alpha}\cup S_{\alpha^2}\subset P_n(4)$);
   \item[b)] or one polynomial square root (it is a point in $(S\cup S_{\alpha}\cup S_{\alpha^2})\setminus C$);
   \item[c)] or three polynomial square roots (it is a point on the curve $C$).
   \end{itemize}
\end{corollary}

   It would be interesting to know if the geometric picture carries over to the general setting i.e.
   if the geometry of the set of polynomials of degree $d^r$, which are $r$-th iterative
   powers of polynomials of degree $d$ has a similar description and if one can
   characterize polynomials with multiple iterative polynomial roots. We hope to come back
   to this question in the next future.

\subsection*{{\it Acknowledgements.}}
   We would like to thank the anonymous referees for helpful remarks which
   greatly improved the exposition of the paper.
   The second author was partially supported by
   a MNiSW grant N~N201 388834.

%
%

\end{document}